\documentclass[letterpaper, 11pt, oneside, reqno]{amsart}

\usepackage{framed, multirow}
\usepackage[linesnumbered,ruled]{algorithm2e}
\usepackage{mathrsfs}
\usepackage{amssymb}
\usepackage{amsthm}
\usepackage{amsmath}
\usepackage{amsfonts}
\usepackage{latexsym}
\usepackage{hyperref}
\usepackage{multicol}
\usepackage{graphicx}
\usepackage[margin=1in]{geometry}
\usepackage{float}
\usepackage[all]{xy}
\usepackage{url}
\usepackage{xcolor}
\definecolor{newcolor}{rgb}{.8,.349,.1}

\newtheorem{theorem}{Theorem}

\title{On smooth families of exact forms}

\author{Jes\'us F. Espinoza}
\email{jesus.espinoza@mat.uson.mx}

\author{Rafael Ramos}
\email{rramos@mat.uson.mx}

\address{Departamento de Matem\'aticas, Universidad de Sonora. Hermosillo, Sonora, M\'exico.}

\keywords{$k$-forms, \v{C}ech-de Rham complex, good cover, vector field.}
\begin{document}
\maketitle
\begin{abstract}
For a smooth family of exact forms on a smooth manifold, an algorithm for computing a primitive family smoothly dependent on parameters is given. The algorithm is presented in the context of a diagram chasing argument in the \v{C}ech-de Rham complex. In addition, explicit formulas for such primitive family are presented.
\end{abstract}

\section{Introduction}
In various problems of differential and symplectic geometry, the following fact plays a fundamental role \cite{McDuff}.

\begin{theorem}\label{AAA}
Let $\{ \omega_x\}_{x \in \mathcal{P}}$ be a family of  exact $(k+1)$-forms on a smooth manifold $M$, which smoothly depends on a parameter $x$ that takes values in an open set $\mathcal{P} \subseteq \mathbb{R}^n$. There then exists a smooth family of $k$-forms $\{ \tau_{x}\}_{x \in \mathcal{P}}$ on $M$ such that $d \tau_{x} = \omega_{x}$ for each $x \in \mathcal{P}$.
\end{theorem}

Even when this fact is commonly used, it is difficult to find a complete proof of Theorem \ref{AAA} in the existent bibliography. Furthermore, in cases when a complete proof is given, the algorithm for find such primitives or collating formulas remain more or less obscure because of the number of steps involved. For example, for a general case, a sketch of the proof of this theorem is given in \cite{Gotay}. A more detailed proof for a manifold $M$ of a finite type is given in \cite{Marcut}. In addition, in the case when $M$ is compact, Theorem \ref{AAA} can be proven by using the Hodge theory \cite{Ibort}, \cite{McDuff}. 

In the present article, for an arbitrary manifold $M$ (without boundary) and by using elementary tools, we give a complete constructive proof of Theorem \ref{AAA} consisting of a shorter algoritm than the proofs shown in \cite {Gotay} and \cite{Marcut}. In addition, we get explicit collating formulas for the primitive family. Even when only elementary arguments are necesary for our proof, we decided to present it in the context of a diagram chasing argument in the \v{C}ech-de Rham complex \cite{Bott} because in this way the proof is easier to follow.

Finally, in Section \ref{Section:application} we give an application of Theorem \ref{AAA} proving Theorem \ref{beta3}, which states the following: {\it Let $(M,\mathcal{F})$ be a foliated manifold such that the foliation $\mathcal{F}$ is given by a locally trivial submersion $p: M \to B$. If $H^r_{DR}(p^{-1}(b))=0$ for all $b \in B$ and an integer $r$, then $H^r(M, \phi^0(\mathcal{F}))=0$}.

Now we will proceed to describe the steps of the algorithm, which is divided into two cases. In the first case, we construct a family of exact $1$-forms and in the second one a family of exact $k$-forms such that $k \geq 2$ is constructed.

Before the algorithm, we prove the assertion of Theorem \ref{AAA} in the local case by applying the standard argument of the Poincar\'e lemma for forms that are smoothly dependent on parameters. In this way, associated with a countable good cover $\{U_\alpha \}_{\alpha \in J}$ of $M$, for each fixed $x$  we get a collection $\{ \tau^\alpha_x \}_{\alpha \in J}$ of $k$-forms (i.e, an element in $\Pi\Omega^k(U_\alpha)$) such that for each $\alpha$ the family $\{ \tau_x^\alpha  \}_{x \in\mathcal{P}}$ depends smoothly on the parameter $x$. 

Now we will describe the case for constructing a family of exact $0$-forms given a family of $1$-forms as in Theorem \ref{AAA}. The algorithm starts (Step 1) with the \v{C}ech-de Rham complex associated with the good covering of $M$, and by using a commutative square in such a complex, we construct a family of constants $C^{\alpha_0 \alpha_1}_x$ that smoothly depends on the parameter $x$, and  this family is defined for each $\alpha_0$, $\alpha _1$ in $J$ such that the intersection of $U_{\alpha_0}$ and $U_{\alpha_1}$ is not empty. In Step 2, by using the hypothesis of our main theorem, we define auxiliary constants $\widetilde{C}^\alpha_x$, which do not necessarily smoothly depend on the parameter $x$, but this allows us to smoothly extend the definition of the constants $C^{\alpha_0 \alpha_1}_x$ for all $\alpha_0$, $\alpha_1$ in $J$. In the Step 3, we define a new smooth family of $0$-forms $\{ \widetilde{\tau}_x^\alpha \}_{\alpha \in J}$ on the good cover $\{ U_{\alpha} \}_{\alpha \in J }$ by fixing any element $\alpha_0$ in $J$ and adding the constant $-C_x^{\alpha_0 \alpha}$ to the $0$-form $\tau_x^\alpha$ for each $\alpha$.

In the last step (Step 4), by using a diagram chasing argument in the \v{C}ech-de Rham complex, we prove that the family $\{ \widetilde{\tau}_x^\alpha  \}$ defined in Step 3 determines a family of 0-forms stated in Theorem \ref{AAA}.

Let us describe now the steps for constructing a smooth family of $k$-forms $\{ \tau_{x}\}_{x \in \mathcal{P}}$ for a given family $\{ \omega_x\}_{x \in \mathcal{P}}$ of exact $(k+1)$-forms. In the first step (Step 1), by using a diagram chasing argument in the \v{C}ech-de Rham complex associated with the good cover of $M$ for each fixed $x$ in $\mathcal{P}$, we define an element $\{ a_x^\alpha \}_{\alpha \in J}$ in $\Pi\Omega^{k-1}(U_\alpha)$. Such an element depends on the element $\{ \tau_x^\alpha \}_{\alpha \in J}$ and depends on the primitives of the collection of $(k+1)$-forms $\{ w_x \}$ on $M$. When $\alpha$ in $J$ is fixed, the family $\{ a_x^\alpha \}_{x\in \mathcal{P}}$ is not necessarily smooth on $x$. In Step 2, however, for each fixed $x$, we define an element $\{ g_x^{\alpha_0 \alpha_1}  \}_{\alpha_0 \alpha_1 \in J}$ in $\Pi\Omega^{k-1}(U_{\alpha_0 \alpha_1})$ that depends on the element $\{ a_x^\alpha  \}_{\alpha \in J}$ and that for each fixed pair $\alpha_0$, $\alpha_1$ the family $\{ g_x^{\alpha_0 \alpha_1} \}_{x \in \mathcal{P}}$ smoothly depends on $x$. 

In Step 3, we prove that the element $\{ g_x^{\alpha_0 \alpha_1} \}_{\alpha_0 \alpha_1 \in J}$ is a cocycle in $\Pi\Omega^{k-1}(U_{\alpha_0 \alpha_1})$. By then using the exactness of the \v{C}ech-de Rham complex, we define an element $\{ G_x^\alpha \}_{\alpha \in J}$ in $\Pi\Omega^{k-1}(U_\alpha)$, which depends on such a cocycle and we prove that for each fixed $\alpha$ in $J$ the family $\{ G_x^\alpha \}_{\alpha \in J}$ smoothly depends on $x$. In Step 4, for each $x$ we define a new element $\{ \widetilde{\tau}_x^\alpha \}_{\alpha \in J}$ in $\Pi\Omega^k(U_\alpha)$ by adding $\tau_x^\alpha$ to $d G_x^\alpha$ for each $\alpha$ in $J$. Thus for each $\alpha$ the family $\{ \widetilde{\tau}_x^\alpha \}_{x \in \mathcal{P}}$ smoothly depends on $x$. We then prove by a diagram chasing argument in the \v{C}ech-de Rham complex that the family $\{ \widetilde{\tau}_x^\alpha \}_{\alpha \in J}$ determines the family of $k$-forms requested in Theorem \ref{AAA}.

%We recall that  the family 
%$\{ \omega_x\}_{x \in \mathcal{P}}$ 
%is smoothly depending on the parameter $x$ in $\mathcal{P}$ if for each colection of smooth vectorial fields $X_{1}$,
%$\dots$ ,$X_{k}$  in $\Xi (M)$ the function $F_{X_{1}, \dots ,X_{k} }: \mathcal{P} \times M \rightarrow R$ defined as
%$$F_{X_{1}, \dots ,X_{k} }(x,m) = \tau_{x}( X_{1}(m), \dots ,X_{k}(m) )$$
%is $C^{\infty}$ and where $\mathcal{P} \times M$ have the usual product differential structure.
%In case of $0-$forms we ask that the function 
%$F: \mathcal{P} \times M \rightarrow R$ defined as $F(x,m)=\tau_{x}(m)$ is $C^{\infty}$.
%\\

\section{The local case}\label{Section:local.case}
Let $M$ be a smooth manifold of dimension $d$. We recall that  an open cover $\{ U_{\alpha} \}$ of $M$ is called a {\it good cover} if all non-empty finite intersections $U_{\alpha_0} \cap \dots \cap U_{\alpha_p}$ are diffeomorphic to $\mathbb{R}^d$. Notice that all smooth manifold $M$ has a good cover \cite{Bott}. Actually, every smooth manifold $M$ that is second countable has a countable good cover. Indeed, by \cite{Warner} $M$ has a countable open cover $\{ V_r \}$ consisting of sets with compact closures. Let $\{U_s \}$
be an open good cover for $M$. For each $r$, $\bar{V_r}$ is covered by a finite number of elements in $\{ U_s \}$. By taking the union of such finite collections when $r$ varies in all the elements of $\{ V_r \}$, we get a countable coordinate good cover $\{ U_\alpha \}$ for the manifold $M$.

We will denote by $U_{\alpha_0 \cdots \alpha_p}$ the intersection $U_{\alpha_0} \cap \cdots \cap U_{\alpha_p}$ of the open sets $U_{\alpha_0}, \ldots, U_{\alpha_p}$ in the good cover. Furthermore, we denote by $\Omega^q(M)$ the $\mathbb{R}$-vector space of $q$-forms on $M$ and define $\Pi \Omega(U_{\alpha_0, \dots, \alpha_p})$ as the product of the collection $\{ \Omega^q(U_{\alpha_0, \dots, \alpha_p}) \}_{ (\alpha_0, \dots, \alpha_p) \in J^p }$.

On the other hand, let $\{ \omega_x \}_{ x \in \mathcal{P} }$ be a family of closed $(k+1)$-forms ($k\geq 0$) over the $d$-manifold $M$, which smoothly depends on the parameter $x$ in $\mathcal{P} \subseteq \mathbb{R}^n$. Thus, by the Poincar\'e lemma (as in \cite{Michor}) for each $x \in \mathcal{P}$ and any chart $(U, u:U \to \mathbb{R}^d)$ on $M$ with $U$ an open contractible subset of $M$, there exists a primitive $\eta_x$ of $\omega_x$ given by the formula
\begin{equation}\label{C}
\eta_x = \int_0^1 \dfrac{1}{t} \lambda_t^* \iota_I \omega_x dt,
\end{equation}
where $\lambda : \mathbb{R} \times U \rightarrow U$ is defined by $\lambda(t,m) = \lambda_t(m) = u^{-1}(tu(m))$, $I$ is the vector field $I(m)=m$ and $\iota_I$ is the insertion operator (cf. \cite{Michor}). It is clear that the $k$-forms  $\eta_x$ in (\ref{C}) smoothly depend on $x$. Therefore, we have a smooth family $\{ \eta_x \}_{x \in \mathcal{P}} $ of primitives.

Let $\{ U_\alpha \}_{\alpha \in J}$ be a fixed good cover of the manifold $M$ where the family of indexes $J$ is a countable ordered set. It follows that for each $\alpha$ in $J$ there exists a collection $\{\tau^{\alpha}_{x}\}_{x \in \mathcal{P}}$ with the following properties: If $k=0$, then $\tau^{\alpha}_{x} \in\Omega^{0}(U_{\alpha})$ for each $x$ in $\mathcal{P}$, $d\tau_x^\alpha = \omega_x |_{ U_\alpha}$ for each $x$ in $\mathcal{P}$ and the family ${ \{ \tau^{\alpha}_{x} \} }_{x \in \mathcal{P}}$ is smoothly dependent on $x$ in $\mathcal{P}$, i.e. the function $F^\alpha : U_\alpha \times \mathcal{P} \rightarrow \mathbb{R}$ defined by $F^\alpha(m,x) = \tau^{\alpha}_{x}(m)$ is  $\mathcal{C}^{\infty}$; otherwise ($k \geq 1$), $\tau^{\alpha}_{x} \in\Omega^{k}(U_{\alpha})$ for each $x$ in $\mathcal{P}$, $d\tau_x^\alpha = \omega_x |_{ U_\alpha}$ for each $x$ in $\mathcal{P}$, and the family ${ \{ \tau^{\alpha}_{x} \} }_{x \in \mathcal{P}}$ is smoothly dependent on the parameter $x$ in $\mathcal{P}$. This means that for each $\alpha$ and for each collection of smooth vector fields $X_1, \dots, X_k$ on $U_{\alpha}$ the function $F^\alpha_{X_1, \dots, X_k} : U_\alpha \times \mathcal{P} \rightarrow \mathbb{R}$ defined by $F^\alpha_{X_1, \dots, X_k}(m,x) = ( \tau^{\alpha}_{x}(X_1, \dots, X_k))(m)$ is  $\mathcal{C}^{\infty}$. 

We will use such functions $F^\alpha : U_\alpha \times \mathcal{P} \rightarrow \mathbb{R}$ and $F^\alpha_{X_1, \dots, X_k} : U_\alpha \times \mathcal{P} \rightarrow \mathbb{R}$, in our algorithm in the next section.

\section{An algorithm to construct a smooth primitive family}
In this section, we present an algorithm to construct a smooth family of $k$-forms $\{ \tau_{x}\}_{x \in \mathcal{P}}$ on a manifold $M$ such that $d \tau_{x} = \omega_{x}$ for each $x \in \mathcal{P}$ for a given family $\{ \omega_x\}_{x \in \mathcal{P}}$ of exact $(k+1)$-forms, which smoothly depends on a parameter $x$ in an open set $\mathcal{P} \subseteq \mathbb{R}^n$. 

The algorithm is divided into two cases. The first one is when $\{ \omega_x \}_{x \in \mathcal{P}}$ is a family of exact $1$-forms, and the second case is for exact $k$-forms such that $k \geq 2$.

\subsection{Constructing a smooth family of exact 1-forms}

Let $M$ be a smooth $d$-manifold, and let $\{ U_\alpha \}_{\alpha \in J}$ be a fixed good cover of $M$, with $J$ a countable ordered set. If $\{ \omega_x \}_{x \in P}$ is a smooth family of exact $1$-forms on $M$, then there exists a family $\{\eta_{x}\}_{x \in \mathcal{P}}$ of $0$-forms in $\Omega^{0}(M)$ with the property that $d \eta_{x} = \omega_{x}$ for each $x$ in $\mathcal{P}$.

For each $\alpha$ in $J$, let $\{\tau^{\alpha}_{x}\}_{x \in \mathcal{P}}$ be the collection of 0-forms defined in Section \ref{Section:local.case}, and let $F^\alpha : U_\alpha \times \mathcal{P} \rightarrow \mathbb{R}$ be the $\mathcal{C}^{\infty}$-map defined by $F^\alpha(m,x) = \tau^{\alpha}_{x}(m)$.

\subsection*{Step 1. Defining a smooth family of constants}

Let us associate the \v{C}ech-de Rham complex with the good cover $\{ U_{\alpha} \}$ of the manifold $M$; see \cite{Bott}. We have the following commutative diagram with exact rows:
\begin{gather}
\begin{aligned}
\xymatrix{
                                                                             &
                                                                             & 
                                                                             &
                                                                             &
                                                                             &           
                                                                             \\ 
        0 \ar[r]                                                             &
        \Omega^1(M) \ar[r]^-{r}\ar[u]                                         &
        \Pi \Omega^{1}(U_{\alpha_{0}}) \ar[r]^-{\delta_1}\ar[u]               &
        \Pi\Omega^{1}(U_{\alpha_{0} \alpha_1 }) \ar[r]^-{\delta_1}\ar[u]      &
        \Pi \Omega^{1}(U_{\alpha_{0}\alpha_1 \alpha_2}) \ar[r]\ar[u]         &
        \cdots                                                               \\
        0 \ar[r]                                                             &
        \Omega^0(M) \ar[r]^-{r} \ar[u]^-{d_0}                                  &
        \Pi \Omega^{0}(U_{\alpha_{0}}) \ar[r]^-{\delta_0}\ar[u]^-{d_o}         &
        \Pi\Omega^{0}(U_{\alpha_{0}\alpha_1}) \ar[u]_{d_o}^{(\mathrm{I}) \ \ \ \ \ \ \ \ \ \ } \ar[r]^-{\delta_0} &
        \Pi\Omega^{0}(U_{\alpha_{0}\alpha_1 \alpha_2}) \ar[r] \ar[u]^-{d_0}   &
        \cdots
         }
\end{aligned}
\label{1}
\end{gather}
%%%%%%%%%%%%%%%%%%%%%%%%%%%%%%%%%%%%%%%%%%%%%%%%%%%%%%%%%%%%%%%%%%%%%%%%%%%%%%%%%%%%%%%%%%%%%%%%%%%%%%%%%%%%%%%%%%%%%

In every vertical arrow, $d$ denotes the differential operator induced by the usual differential operator of the de Rham complex. Moreover, for every $q \in \mathbb{Z}$ let us introduce the coboundary operator 
\[ \delta : \Pi\Omega^{q}(U_{\alpha_{0} \dots \alpha_p }) \rightarrow \Pi\Omega^{q}(U_{\alpha_{0} \dots \alpha_{p+1} }) \]
defined as follows: If $\omega$ belongs to $\Pi\Omega^{q}(U_{\alpha_{0} \dots \alpha_p })$, then $\omega$ has components $\omega_{\alpha_0 \dots \alpha_p}$ in $\Omega^{q}(U_{\alpha_{0} \dots \alpha_p })$, and the components of $\delta \omega$ are given by
\[ (\delta \omega)_{\alpha_0 \dots \alpha_{p+1}} = \sum_{i=0}^{p+1}(-1)^i \omega_{\alpha_0 \dots \hat \alpha_i \dots \alpha_{p+1}}\mid_{U_{\alpha_0 \dots \alpha_i \dots \alpha_{p+1}}}.\]
Now the family $\{  \tau^{\alpha}_{x}  \}_{\alpha \in J}$ is an element in the $\mathbb{R}$-vector space $\Pi \Omega^{0}(U_{\alpha_{0}})$ for each fixed  $x \in \mathcal{P}$. Thus, by evaluating the element  $\{  \tau^{\alpha}_{x}  \}_{\alpha \in J}$ in the diagram (\ref{1}), we get

\[ \xymatrix{
\{ d_{0}\tau^{\alpha}_{x}   \}_{\alpha \in J} \ar[r]	&	\delta_1 \{ d_0 \tau_x^\alpha \}_{\alpha \in J} = \{ 0 \}	\\
\{  \tau^{\alpha}_{x}  \}_{\alpha \in J} \ar[r] \ar[u]_{\ \ \ \ \ \ \ \ \ \ \ (\mathrm{I})}		&	\delta_0 \{\tau_x^\alpha \}_{\alpha \in J}	\ar[u] \ar[r]	&  \{ 0 \},
}
\]
where $\delta_1 \{ d_0 \tau_x^\alpha     \}_{\alpha \in J} = \{ 0 \}$ because
\[ (\delta_1 \{ d_0 \tau_x^\alpha     \}_{\alpha \in J})_{\alpha_0 \alpha_1} = d_{0}\tau^{\alpha_{1}}_{x} |_{U_{\alpha_{0} \alpha_{1}}} -  d_{0}\tau^{\alpha_{0}}_{x} |_{U_{\alpha_{0}\alpha_{1}}} = w_x |_{U_{\alpha_{0} \alpha_{1}}} -  w_x |_{U_{\alpha_{0} \alpha_{1}}} = 0.\]

We remember that the function $F^\alpha: U_\alpha \times \mathcal{P} \rightarrow \mathbb{R}$ given by $F^\alpha(m,x)=\tau_x^\alpha(m)$ is $\mathcal{C}^\infty$ for all $\alpha$ in $J$. For each fixed $x$ in $\mathcal{P}$ the function $F^\alpha(  -  , x): U_\alpha \rightarrow \mathbb{R}$ then belongs to  $\mathcal{C}^\infty(U_\alpha)$ for all $\alpha$ in $J$. Therefore, for each fixed $x$ in $\mathcal{P}$ the function $(\delta_0 \{   \tau^\alpha_x \}_{\alpha \in J}      )_{\alpha_0 \alpha_1}$ is in $\mathcal{C}^\infty(U_{\alpha_0 \alpha_1})$ for all $U_{\alpha_0 \alpha_1} \neq 0$ because
\begin{align*}
(\delta_0 \{   \tau^\alpha_x \}_{\alpha \in J}      )_{\alpha_0 \alpha_1} &= \tau^{\alpha_1}_x|_{U_{\alpha_0 \alpha_1}}- \tau^{\alpha_0}_x|_{U_{\alpha_0 \alpha_1}} \\
&= F^{\alpha_1}( - , x)|_{U_{\alpha_0 \alpha_1}}- F^{\alpha_0}( - , x)|_{U_{\alpha_0 \alpha_1}} \in \mathcal{C}^\infty(U_{\alpha_0 \alpha_1}).
\end{align*}

Since the above diagram commutes, however, we get $d_0((\delta_0 \{   \tau^\alpha_x \}_{\alpha \in J}      )_{\alpha_0 \alpha_1}) = 0$ on the open connected  $U_{\alpha_0 \alpha_1}$. On the other hand, for every smooth function from a connected manifold into another manifold such that the differential is zero, the map must be constant in the domain; see \cite{Warner}. Therefore, we have for each $x \in \mathcal{P}$ that the function $(\delta_0 \{   \tau^\alpha_x \}_{\alpha \in J}      )_{\alpha_0 \alpha_1}$ is a constant function that only depends on the indexes $\alpha_0$, $\alpha_1$ and the $x$ chosen.

For each $x$, we define a special family of  constants $ \{ C^{\alpha_0 \alpha}_x \}$, which smoothly depends on $x$.

We define for each $\alpha_0, \alpha_1 \in J$ such that $U_{\alpha_0 \alpha_1} \neq 0$ and for each fixed $x \in \mathcal{P}$ the constants in $\mathbb{R}$ given by 
\begin{equation}\label{A}
C_x^{\alpha_0 \alpha_1} = (\delta_0 \{   \tau^\alpha_x \}_{\alpha \in J}      )_{\alpha_0 \alpha_1} .
\end{equation}
We note that when we let the parameter $x \in \mathcal{P}$ vary we have that $C_x^{\alpha_0 \alpha_1} \in C^\infty(\mathcal{P})$
because the function $C_{(-)}^{\alpha_0 \alpha_1} = F^{\alpha_1}(*, - )-F^{\alpha_0}(*,- ): \mathcal{P} \rightarrow \mathbb{R}$ is $\mathcal{C}^\infty$, and it does not depend on the point $*$ chosen in $U_{\alpha_o \alpha_1}$.

\subsection*{Step 2. Extending the definition of the smooth constants}

Since $d_0 \{ \eta_x \vert_{U_\alpha}  \}_{\alpha \in J} = d_0 \{\tau_x^\alpha \}_{\alpha \in J}$ implies that $d_0( \tau_x^\alpha - \eta_x \vert_{U_\alpha} ) = 0$ for all $\alpha$ in $J$, then $\tau_x^\alpha - \eta_x \vert_{U_\alpha}$ is a constant for each fixed $x \in \mathcal{P}$. We define for each $\alpha \in J$ and for each fixed $x \in P$ the constant 
\[ \widetilde{C}_x^\alpha = \tau_x^\alpha - \eta_x \vert_{U_\alpha}.\]
Note that $\widetilde{C}_x^\alpha$ is not necessarily a smooth function when the parameter $x$ in $\mathcal{P}$ varies. Thus, we get 
\[ \tau_x^\alpha = \eta_x \vert_{U_\alpha} + \widetilde{C}_x^\alpha, \]
and then for all $\alpha_0$, $\alpha_1$ in $J$ such that $U_{\alpha_0 \alpha_1} \neq 0$, we have 
\begin{align}
\tau_x^{\alpha_1} \vert_{U_{\alpha_0 \alpha_1}} - \tau_x^{\alpha_0} \vert_{U_{\alpha_0 \alpha_1}} & = (\delta_0 \{ \tau_x^\alpha \}_{\alpha \in J})_{\alpha_0 \alpha_1} \nonumber \\
& = (\delta_0 \{  \eta_x \vert_{U_\alpha} + \widetilde{C}_x^\alpha    \}_{\alpha \in J})_{\alpha_0 \alpha_1} \nonumber\\
& =(\delta_0 \{  \widetilde{C}_x^\alpha   \}_{\alpha \in J})_{\alpha_0 \alpha_1} \nonumber \\
& = \widetilde{C}_x^{\alpha_1}  - \widetilde{C}_x^{\alpha_0}. \label{B}
\end{align}
From (\ref{A}) and (\ref{B}) we get that 
\[ C_x^{\alpha_0 \alpha_1} =  \widetilde{C}_x^{\alpha_1}  - \widetilde{C}_x^{\alpha_0} \]
for all $\alpha_0$, $\alpha_1$ in $J$ such that $U_{\alpha_0 \alpha_1} \neq 0$.

Now we will generalize the definition of the constants $C_x^{\alpha_0 \alpha_1}$ for any $\alpha_0$, $\alpha_1$ in $J$, even if $U_{\alpha_0 \alpha_1 }$ is empty. We define
\[ C_x^{\alpha_0 \alpha_1} =  \widetilde{C}_x^{\alpha_1}  - \widetilde{C}_x^{\alpha_0} \]
for all $\alpha_0$, $\alpha_1$ in $J$.

Next, we will check that $C_x^{\alpha_0 \alpha} \in \mathcal{C}^\infty(\mathcal{P})$ for all $\alpha_0$, $\alpha$ in $J$. We only need check on the case when $U_{\alpha_0 \alpha}$ is empty. We choose any point $m_0$ in $U_{\alpha_0}$. We take a path $\gamma : [0,1] \rightarrow M$ from $m_0$ to any fixed point $m_1$ in $U_{\alpha}$. Since $[0,1]$ is compact, we can choose a finite collection $U_{\alpha_0}, U_{\alpha_1}, \ldots, U_{\alpha_{k-1}}, U_{\alpha_k} = U_{\alpha}$ such that $U_{\alpha_i} \cap U_{\alpha_{i+1}} \neq \emptyset$ for each $i$. We will prove the case $k=2$, and the general case follows by induction.

We suppose that $U_{\alpha_0 \alpha_2}$ is empty. Since $U_{\alpha_0 \alpha_1} \neq \emptyset$ and $U_{\alpha_1 \alpha_2} \neq \emptyset$, we have that
$C_x^{\alpha_0 \alpha_1}$, $C_x^{\alpha_1 \alpha_2}$ are in $\mathcal{C}^\infty(\mathcal{P})$, so is their sum. On the other hand,
\[ C_x^{\alpha_0 \alpha_1} + C_x^{\alpha_1 \alpha_2} = (\widetilde{C}_x^{\alpha_1}  - \widetilde{C}_x^{\alpha_0}) + (\widetilde{C}_x^{\alpha_2}  - \widetilde{C}_x^{\alpha_1}) = \widetilde{C}_x^{\alpha_2} - \widetilde{C}_x^{\alpha_0} = C_x^{\alpha_0 \alpha_2}. \]
Therefore, $C_x^{\alpha_0 \alpha_2}$ is in $\mathcal{C}^\infty(\mathcal{P})$.

Thus, we have that $C_x^{\alpha_i \alpha_j}$ is defined for each pair $\alpha_i$, $\alpha_j$ in $J$ as $\widetilde{C}_x^{\alpha_j} - \widetilde{C}_x^{\alpha_i}$ and $C_x^{\alpha_i \alpha_j} \in \mathcal{C}^\infty(\mathcal{P})$.

\subsection*{Step 3. Redefining the family of 0-forms}

By adding the family of constants $\{ C^{\alpha_0 \alpha}_x \}$ that we got in the last step to the family of $0$-forms $\{ \tau^\alpha_x \}$ which we got in Step 1, we define a new family $\{ \widetilde{\tau}^\alpha_x \}$, which has the same properties as the family $\{ \tau^\alpha_x \}$ except that now the $\delta$ operator applied to this new family is zero. We proceed as follows.

We choose any fixed $\alpha_0 \in J$, and we define $\widetilde{\tau}^\alpha_x = \tau^\alpha_x - C^{ \alpha_0 \alpha}_x$ for each $\alpha$ in $J$ and for each $x$ in $\mathcal{P}$. Thus, we have $\widetilde{\tau}^\alpha_x \in \Omega^0(U_\alpha)$ and $d_0 \widetilde{\tau}^\alpha_x = d_0(\tau^\alpha_x - C^{\alpha_0 \alpha}_x) = d_0(\tau^\alpha_x) - d_0(C^{\alpha_0 \alpha}_x) = d_0(\tau^\alpha_x) - 0 = \omega_x|_{U\alpha}$
for each fixed $x$ in $\mathcal{P}$ and for all $\alpha$ in $J$. In addition, the new family $\{ \widetilde{\tau}^\alpha_x \}_{x \in \mathcal{P}}$ is smoothly depending on $x$ in $\mathcal{P}$ like the family $\{ \tau^\alpha_x \}_{x \in \mathcal{P}}$. Now, however, the family $\{ \widetilde{\tau}^\alpha_x \}_{x \in \mathcal{P}}$ has the additional property that $\delta_0 \{  \widetilde{\tau}^\alpha_x   \}_{x \in \mathcal{P}}=\{ 0 \}$ because
\begin{align*}
\delta_0(\{ \widetilde{\tau}_x^\alpha \}_{\alpha \in J})_{\alpha_i \alpha_j} &= \widetilde{\tau}_x^{\alpha_j} \vert_{U_{\alpha_i \alpha_j}} - \widetilde{\tau}_x^{\alpha_i} \vert_{U_{\alpha_i \alpha_j}} \\
&= (\tau_x^{\alpha_j} \vert_{U_{\alpha_i \alpha_j}}  - C_x^{\alpha_0 \alpha_j}) - (\tau_x^{\alpha_i} \vert_{U_{\alpha_i \alpha_j}}  - C_x^{\alpha_0 \alpha_i}) \\
& =(\tau_x^{\alpha_j} \vert_{U_{\alpha_i \alpha_j}} - \tau_x^{\alpha_i} \vert_{U_{\alpha_i \alpha_j}}) - (C_x^{\alpha_0 \alpha_j} - C_x^{\alpha_0 \alpha_i}) \\
& =(\widetilde{C}_x^{\alpha_j} - \widetilde{C}_x^{\alpha_i}) - (C_x^{\alpha_0 \alpha_j} - C_x^{\alpha_0 \alpha_i}) \mbox{ by (\ref{B}).} \\
& =(\widetilde{C}_x^{\alpha_j} - \widetilde{C}_x^{\alpha_i}) - ((\widetilde{C}_x^{\alpha_j} - \widetilde{C}_x^{\alpha_0}) - (\widetilde{C}_x^{\alpha_i} - \widetilde{C}_x^{\alpha_0})) \\
& =0.
\end{align*}

\subsection*{Step 4. Diagram chasing in the \v{C}ech-de Rham complex}

Since the rows in diagram (\ref{1}) are exact, then the collection $\{ \widetilde{\tau}^\alpha_x \}_{x \in \mathcal{P}}$ determines a collection of global $0$-forms $\{  \tau_x  \}_{\ x \in \mathcal{P}}$ in $\Omega^0(M)$. Such a collection is smoothly dependent on the parameter $x$ in $\mathcal{P}$ and $d_0 \tau_x = \omega_x$ for each $x$ in $\mathcal{P}$. %So we finished with proof for case $k=0$.

In summary for case $k=1$, we define the smooth family of primitives $\tau_x$ by 
\[ \tau_x \vert_{U_\alpha} = \tau_x^\alpha - A_x^\alpha + A_x^{\alpha_{0}} \] 
for each $\alpha \in J$, where the constants $A_x^\alpha$ are defined by 
\[ A_x^\alpha = \tau_x^\alpha -\eta_x \vert_{U_\alpha}\]
for each $\alpha \in J$ and $\alpha_0$ is any fixed element in $J$.

\subsection{Constructing a smooth family of exact \textit{k}-forms}
For any $k \geq 2$, we now construct a smooth family of $k$-forms $\{ \tau_{x}\}_{x \in \mathcal{P}}$ on a manifold $M$ such that $d \tau_{x} = \omega_{x}$ for each $x \in \mathcal{P}$, for a given family $\{ \omega_x\}_{x \in \mathcal{P}}$ of exact $(k+1)$-forms, which smoothly depends on a parameter $x$ in an open set $\mathcal{P} \subseteq \mathbb{R}^n$. 

For each $\alpha$ in $J$, let $\{\tau^{\alpha}_{x}\}_{x \in \mathcal{P}}$ be the collection of $k$-forms defined in Section \ref{Section:local.case}, and for each collection of smooth vector fields $X_1, \dots, X_k$ on $U_{\alpha}$, let $F^\alpha_{X_1, \dots, X_k} : U_\alpha \times \mathcal{P} \rightarrow \mathbb{R}$ be the $\mathcal{C}^{\infty}$-function defined by $F^\alpha_{X_1, \dots, X_k}(m,x) = ( \tau^{\alpha}_{x}(X_1, \dots, X_k))(m)$.

\subsection*{Step 1. Defining an element $\{ a_x^\alpha \}_{\alpha \in J}$ in $\Pi\Omega^{k-1}(U_\alpha)$ via a diagram chasing in the \v{C}ech-de Rham complex}

Associated with the good cover $\{ U_\alpha \}_{\alpha \in J}$ of the smooth manifold $M$, we have a \v{C}ech-de Rham complex \cite{Bott}, which implies that we have the following commutative diagram with exact horizontal and vertical arrows:
\begin{gather}
\begin{aligned}
\xymatrix{
			& \Omega^{k+1}(M)\ar[r]^-{\delta_{k+1}}	
			& \Pi \Omega^{k+1}(U_{\alpha_0}) \ar[r]^-{\delta_{k+1}}	
			& \Pi \Omega^{k+1}(U_{\alpha_{0} \alpha_1}) 
			\\
\ 0 \ar[r]	& \Omega^k(M) \ar[r]^-{r}\ar[u]^-{d_k}_{ \ \ \ \ \ \ \ \ (\mathrm{I})}
			& \Pi \Omega^k(U_{\alpha_0}) \ar[r]^-{\delta_k} \ar[u]^-{d_k}_{ \ \ \ \ \ \ \ \ \ \ (\mathrm{II})} 
			& \Pi \Omega^k(U_{\alpha_0 \alpha_1}) \ar[u]^-{d_k} 
			\\
			&
			& \Pi\Omega^{k-1}(U_{\alpha_0}) \ar[r]^-{\delta_{k-1}} \ar[u]^-{d_{k-1}}_{ \ \ \ \ \ \ \ \ \ \ (\mathrm{III})}
			& \Pi \Omega^{k-1}(U_{\alpha_{0} \alpha_1}) \ar[r]^-{\delta_{k-1}} \ar[u]^-{d_{k-1}}
			& \Pi\Omega^{k-1}(U_{\alpha_{0} \alpha_1 \alpha_2 })
}
\end{aligned}
\label{2}
\end{gather}

By hypothesis, there exists a family $\{\eta_{x}\}_{x \in \mathcal{P}}$ of $k$-forms on $M$  with the property $d \eta_{x} = \omega_{x}$ for each $x$ in $\mathcal{P}$. For each $x \in \mathcal{P}$ the element $\{ \tau_x^\alpha - \eta_x |_{U_\alpha} \}_{\alpha \in J}$ then belongs to $\Pi \Omega^k(U_\alpha)$. This element is not necessarily smoothly dependent on $x$. In addition, since $d \eta_{x} = \omega_{x}$ we get that $\delta_k ( \{ \eta_x |_{U_\alpha} \}_{\alpha \in J}) = \{ 0 \}$ for each $x$ in $\mathcal{P}$.
Thus, 
\[ \delta_k ( \{ \tau^{\alpha}_{x} - \eta_x |_{U_\alpha} \}) = \delta_k ( \{ \tau^{\alpha}_{x} \}) - \delta_k ( \{ \eta_x |_{U_\alpha} \}) = \delta_k ( \{ \tau^{\alpha}_{x} \}) - \{0 \} = \delta_k ( \{ \tau^{\alpha}_{x} \}). \]
Therefore, by evaluating the element $\{ \tau_x^\alpha - \eta_x |_{U_\alpha} \}_{\alpha \in J}$ in diagram (\ref{2}), we get the following commutative diagram
\[\xymatrix{
\{ 0 \}	\\
\{ \tau^{\alpha}_{x} - \eta_x |_{U_\alpha} \} \ar[r]^-{\delta_k} \ar[u]^-{d_k}	&	\delta_k ( \{ \tau^{\alpha}_{x} \})	\\
\{a_x^\alpha \} \ar@{-->}[u]^-{d_{k-1}}_{ \ \ \ \ \ \ \ \ (\mathrm{III})}\ar[r]^-{\delta_{k-1}}	&	\delta_{k-1} \{ a_x^\alpha \}\ar[u]^-{d_{k-1}}	.}
\]

Indeed, since vertical arrows are exact in diagram (\ref{2}) (by the Poincar\'e lemma), there exists an element $\{a_x^\alpha \}_{\alpha \in J}$ in $\Pi \Omega^{k-1}(U_\alpha)$, which is not necessarily smoothly dependent on $x$ such that for each $x \in \mathcal{P} $ we have that
\[ d_{k-1}(\{a_x^\alpha \}_{\alpha \in J}) = \{ \tau^{\alpha}_{x} - \eta_x |_{U_\alpha} \}_{\alpha \in J}. \]
Thus, 
\[ \delta_k d_{k-1}(\{a_x^\alpha \}_{\alpha \in J}) = \delta_k \{ \tau^{\alpha}_{x}  \}_{\alpha \in J}. \]
In addition, since square (III) in diagram (\ref{2}) commutes, we get
\[ d_{k-1} \delta_{k-1}(\{a_x^\alpha \}_{\alpha \in J}) = \delta_k \{ \tau^{\alpha}_{x}  \}_{\alpha \in J}. \]

\subsection*{Step 2. Defining a smooth element $\{  g_x^{\alpha_0 \alpha_1}  \}_{\alpha_0 \alpha_1 \in J}$ in $\Pi\Omega^{k-1}(U_{\alpha_0 \alpha_1})$ depending on $\{ a_x^\alpha \}_{\alpha \in J}$} 

We define an element $\{ g_x^{\alpha_0 \alpha_1} \}_{\alpha_0, \alpha_1 \in J}$ in $\Pi \Omega^{k-1}(U_{\alpha_0 \alpha_1})$
by setting 
\[ g_x^{\alpha_0 \alpha_1} = \delta_{k-1}(\{a_x^\alpha \}_{\alpha \in J}) \]
for each $\alpha_0$, $\alpha_1$ in $J$. We will proceed to prove that for each fixed $\alpha_0$, $\alpha_1$  in $J$ the family $\{ g_x^{\alpha_0 \alpha_1} \}_{x \in \mathcal{P}}$ varies smoothly with respect to the parameter $x$ in $\mathcal{P}$.

Explicitly, we have for each $\alpha$ in $J$ that 
\[ a_x^\alpha =  \int_0^1 \dfrac{1}{t} \lambda_t^* \iota_I (\tau_x^\alpha -\eta_x \vert_{U_\alpha}) dt = \int_0^1 \dfrac{1}{t} \lambda_t^* \iota_I (\tau_x^\alpha) - \int_0^1 \dfrac{1}{t} \lambda_t^* \iota_I (\eta_x \vert_{U_\alpha}).\]
Therefore,
\begin{align*}
g_x^{\alpha_0 \alpha_1} &= (\delta_{k-1} \{ a_x^\alpha  \})_{\alpha_0 \alpha_1} \\
& = a_x^{\alpha_1} \vert_{U_{\alpha_0 \alpha_1}} - a_x^{\alpha_0} \vert_{U_{\alpha_0 \alpha_1}} \\
& = \Big( \int_0^1 \dfrac{1}{t} \lambda_t^* \iota_I (\tau_x^{\alpha_1}) - \int_0^1 \dfrac{1}{t} \lambda_t^* \iota_I (\eta_x \vert_{ U_{\alpha_1}  }) \Big) \vert U_{\alpha_0 \alpha_1} - \Big(\int_0^1 \dfrac{1}{t} \lambda_t^* \iota_I (\tau_x^{\alpha_0}) - \int_0^1 \dfrac{1}{t} \lambda_t^* \iota_I (\eta_x \vert_{ U_{\alpha_0}  }) \Big) \vert U_{\alpha_0 \alpha_1} \\
& = \int_0^1 \dfrac{1}{t} \lambda_t^* \iota_I (\tau_x^{\alpha_1} \vert U_{\alpha_0 \alpha_1} - \tau_x^{\alpha_0} \vert U_{\alpha_0 \alpha_1}),
\end{align*}
where the last expression smoothly depends on $x$. Therefore, we obtain that $\{  g_x^{\alpha_0 \alpha_1}\}_{x \in \mathcal{P}}$ varies smoothly with respect to the parameter $x$ in $\mathcal{P}$ for each fixed $\alpha_0$, $\alpha_1$ in $J$.  

\subsection*{Step 3. Defining a smooth element $\{ G_x^\alpha \}_{\alpha \in J}$ in $\Pi\Omega^{k-1}(U_\alpha)$ on terms of the element $\{ g_x^{\alpha_0 \alpha_1}  \}_{\alpha_0 \alpha_1 \in J}$}

Since 
\[ \delta_{k-1} \{ g_x^{\alpha_0 \alpha_1 } \}_{\alpha_0, \alpha_1 \in J} = \delta_{k-1} \delta_{k-1} \{ a_x^\alpha \}_{\alpha \in J} = \{ 0 \}, \]
then the element $\{ g_x^{\alpha_0 \alpha_1 } \}_{\alpha_0, \alpha_1 \in J}$ is a cocycle on $\Pi\Omega^{k-1}(U_{\alpha_0 \alpha_1})$. Furthermore, because of the exactness of the \v{C}ech-de Rham complex, there exists an element $\{ G_x^\alpha \}_{\alpha \in J}$ in $\Pi \Omega^{k-1}(U_\alpha)$ such that
\[ \delta_{k-1} \{ G_x^\alpha \}_{\alpha \in J} = \{ g_x^{\alpha_0 \alpha_1} \}_{\alpha_0, \alpha_1 \in J}, \]
as is shown on the following diagram:
\begin{gather}
\begin{aligned}
\xymatrix{
\{ 0 \}	\\
\{ \tau^{\alpha}_{x} - \eta_x |_{U_\alpha} \} \ar[r]^-{\delta_k} \ar[u]^-{d_k}	&	\delta_k ( \{ \tau^{\alpha}_{x} \})		\\
\{G_x^\alpha \} \ar@{-->}[r]^-{\delta_{k-1}}	&	\delta_{k-1} \{ a_x^\alpha \}\ar[u]^-{(\mathrm{III}) \ \ \ \ \ \ \ \ \ \ \ \ }^{d_{k-1}}\ar[r]^-{\delta_{k-1}}	&	\{0\}.	}
\end{aligned}
\label{D}
\end{gather}

Moreover, for each $\alpha$ in $J$ the family $\{ G_x^\alpha \}_{\alpha \in J}$ can be chosen such that for each $\alpha$ in $J$ the family  $\{ G_x^\alpha \}_{x \in \mathcal{P}}$ varies smoothly on $x$. This is a consequence of this family being defined by using a partition of the unity and by using sums of elements $g_x^{\beta\alpha}$ in $\Omega^{k-1}(U_{\beta \alpha})$. In fact for each $\alpha$ in $J$ we can define
\[ G_x^\alpha = \sum_{\beta \in J} \rho_{\beta}g_x^{\beta \alpha}, \]
where $\{ \rho_\beta: M \rightarrow \mathbb{R} \}_{\beta \in J}$ is a partition of unity subordinate to the good cover $\{ U_\alpha \}_{\alpha \in J}$ of $M$ as in the proof of exactness of the generalized Mayer-Vietoris  sequence in \cite{Bott}. In order to check explicitly that the family $\{ G_x^\alpha \}_{x \in \mathcal{P}}$ varies smoothly on $x$, let $X_1$, $\dots$, $X_{k-1}$ be smooth vector fields on $U_{\alpha}$. Thus, the function
\[ F^\alpha_{X_1, \dots X_{k-1}}: U_\alpha \times \mathcal{P} \rightarrow \mathbb{R} \]
defined as
\[ F^\alpha_{X_1, \dots X_{k-1}}(m,x) = (G_x^\alpha)_m((X_1)_m, \dots, (X_{k-1})_m ) \]
is such that
\begin{align}
F^\alpha_{X_1, \dots X_{k-1}}(m,x) & = (G_x^\alpha)_m((X_1)_m, \dots, (X_{k-1})_m ) \nonumber \\
& =(\sum_{\beta \in J} \rho_{\beta}g_x^{\beta \alpha})_m((X_1)_m, \dots, (X_{k-1})_m) \nonumber \\
& =\sum_{\beta \in J} \rho_{\beta}(m) (g_x^{\beta \alpha})_m((X_1)_m, \dots, (X_{k-1})_m) \nonumber \\
& =\sum_{\beta \in J} \rho_{\beta}(m) (g_x^{\beta \alpha})_m ((X_1 \vert_ {U_{\beta \alpha}})_m, \dots, (X_{k-1} \vert_{U_{\beta \alpha}})_m) \nonumber \\
& =\sum_{\beta \in J} \rho_{\beta}(m) F^{\beta \alpha}_{X_1 \vert_ {U_{\beta \alpha}}, \dots, X_{k-1} \vert_{U_{\beta \alpha}}}(m,x) \label{3}
\end{align}
where the functions
\[ F^{\beta \alpha}_{X_1 \vert_ {U_{\beta \alpha}}, \dots, X_{k-1} \vert_{U_{\beta \alpha}}} : U^{\beta \alpha} \times \mathcal{P} \rightarrow \mathbb{R} \]
defined for each $\beta$ as
\[ F^{\beta \alpha}_{X_1 \vert_ {U_{\beta \alpha}}, \dots, X_{k-1} \vert_{U_{\beta \alpha}}}(m,x) = (g_x^{\beta \alpha})_m( (X_1 \vert_ {U_{\beta \alpha}})_m, \dots, (X_{k-1} \vert_{U_{\beta \alpha}})_m   ) \]
are $\mathcal{C}^\infty$ because for each $\beta$ the family $\{  g_x^{\beta \alpha}\}_{x \in \mathcal{P}}$ is smooth on $x$. Since the  sum (\ref{3}) is finite, we get that the function $F^\alpha_{X_1, \dots, X_{k-1} }$ is $\mathcal{C}^\infty$, so the family $\{ G_x^\alpha \}_{x\in \mathcal{P}}$ varies smoothly on $x$. 

\subsection*{Step 4. Diagram chasing in the \v{C}ech-de Rham complex}

From diagram (\ref{D}) and since the square $(\mathrm{III})$ in diagram (\ref{2}) commutes, we have that 
\[ \delta_k( \{ d_{k-1} G_x^{\alpha}  \}_{\alpha \in J}  ) = \delta_k( \{ \tau_x^\alpha \}_{\alpha \in J} ), \]
where $\{ d_{k-1} G_x^{\alpha}  \}_{\alpha \in J}$ is in $\Pi \Omega^k(U_{\alpha})$.

By using the element $\{ \tau_x^\alpha \}_{\alpha \in J}$ in $\Pi \Omega^k(U_{\alpha})$, we define a new element $\{ \widetilde{\tau}_x^\alpha \}_{\alpha \in J}$ in $\Pi \Omega^k(U_{\alpha})$ by setting
\[ \widetilde{\tau}_x^\alpha = \tau_x^\alpha - d_{k-1} G_x^\alpha \]
for each $\alpha$ in $J$. Therefore, the element $\{ \widetilde{\tau}_x^\alpha \}_{\alpha \in J}$ is such that 
\[ \delta_k( \{ \widetilde{\tau}_x^\alpha \}_{\alpha \in J} ) = \{  0 \} \]
and
\[ d_k \{ \widetilde{\tau}_x^\alpha \}_{\alpha \in J} = d_k \{ \tau_x^\alpha \}_{\alpha \in J}. \]
Therefore, we obtain the following commutative diagram:
\[ \xymatrix{
	&	\{ d \tau^{\alpha}_{x}  \} \ar[r]^-{\delta_{k+1}}	&	\delta_{k+1} ( \{ d \tau^{\alpha}_{x}  \} ) = \{ 0 \}	\\
\{ \tau_{x} \}\ar@{-->}[r]^-{r}	&	\{ \widetilde{\tau}^{\alpha}_{x} \} \ar[r]^-{\delta_k} \ar[u]^-{d_k}_{ \ \ \ \ \ \ \ \ \ \ (\mathrm{I})}	&	\{ 0 \} \ar[u]^-{d_k} \ar[r]^-{\delta_k}	&	\{ 0 \}} 
\]

Since the bottom row is exact, there then exists a family $\{ \tau_x \}_{x \in \mathcal{P} }$ in $\Pi \Omega^1(M)$ that smoothly depends on $x$ and such that
\[ r ( \{ \tau_x \}_{x \in \mathcal{P} } ) = \{ \widetilde{\tau}_x^\alpha \}_{\alpha \in J}. \]
But $r : \Pi \Omega^k(M) \rightarrow  \Pi \Omega^1(U_{\alpha})$ in the generalized Mayer-Vietoris sequence is by definition the restriction  \cite{Bott}. Since square (I) from diagram (\ref{2}) commutes, we get

\begin{gather}
\begin{aligned}
\xymatrix{
d_k \{ \tau_x \}_{x \in \mathcal{P} } \ar[r]^-{r}	&	\{ d \tau^{\alpha}_{x} = \omega_x |_{U_\alpha}  \} \ar[r]^-{\delta_{k+1}}	&	\{ 0 \}	\\
\{ \tau_x \}_{x \in \mathcal{P} } \ar[r]^-{r} \ar[u]^-{d_k}_{ \ \ \ \ \ \ \ \ \ \ \ (\mathrm{I})}	&	\{ \widetilde{\tau}^{\alpha}_{x} \} \ar[r]^{\delta_k} \ar[u]^{d_k} &	\{ 0 \}	}
\end{aligned}
\label{5}
\end{gather}

Diagram (\ref{5}) means that 
\[ r ( \{d_k \tau_x \}_{x \in \mathcal{P}}  ) = \{ \omega_x |_{U_\alpha} \}_{\alpha \in J}, \]
however, or equivalently
\[ ( d_k \tau_x) |_{U_\alpha} = \omega_x |_{U_\alpha} \]
for each $\alpha$ in $J$. Since $\{ U_\alpha \}_{\alpha \in J}$ is a cover for $M$, we have that
\[ d_k \tau_x  = \omega_x \]
for each $x$ in $\mathcal{P}$. Therefore, $\{\tau_x \}_{x \in \mathcal{P}}$ is the required family stated in the theorem. 
\bigskip

In summary, for the case $k \geq 2$ we define the family of smooth primitives $\tau_x$ by
\[ \tau_x \vert_{U_\alpha} = \tau_x^\alpha - d_{k-1}( \ \sum_{\beta \in J}\rho_\beta( a_x^\alpha \vert_{U_{\beta \alpha}} -  a_x^\beta \vert_{U_{\beta \alpha}}   ) ), \]
for each $\alpha \in J$, where $a_x^\alpha$ is the element in $\Omega^{k-1}(U_\alpha)$ defined by
\[ a_x^\alpha =  \int_0^1 \dfrac{1}{t} \lambda_t^* \iota_I (\tau_x^\alpha -\eta_x \vert_{U_\alpha}) dt \]
and where $\{ \rho_\beta: M \rightarrow \mathbb{R} \}_{\beta \in J}$ is a partition of unity subordinate to the good cover $\{ U_\alpha \}_{\alpha \in J}$ of $M$.

\section{Applications}\label{Section:application}
We concluded with an application in proving the following theorem.

\begin{theorem}\label{beta3}
Let $(M,\mathcal{F})$ be a foliated manifold such that the foliation $\mathcal{F}$ is given by a locally trivial submersion $p: M \to B$. If $H^r_{DR}(p^{-1}(b))=0$ for all $b \in B$ and an integer $r$, then $H^r(M, \phi^0(\mathcal{F}))=0$. 
%Let $(M,\mathcal{F})$ be a foliated manifold and $r$ some integer. Assume that the foliation $\mathcal{F}$ is given by a submersion $M \xrightarrow{p} B$ locally trivial. Suppose that $H^r_{DR}(p^{-1}(b))=0$ for all $b$ in $B$. Then $H^r(M, \phi^0(\mathcal{F}))=0$. 
\end{theorem}

\begin{proof} 
Since $M \xrightarrow{p} B$ is a locally trivial bundle, then there exists a countable good cover $\{U_{\alpha} \}_{\alpha \in J}$ of $B$ such that for each $U$ in  $\{U_{\alpha} \}_{\alpha \in J}$ there exists a homeomorphism $\varphi_{U}$ such that the following diagram commutes:
\[ \xymatrix{ p^{-1}(U) \ar[r]^-{\varphi_U} \ar[d]_-{p} & U \times F \ar[d]^-{\pi_1} \\
			U\ar[r]^-{\mathrm{Id}}	& 	U.
			}
\]

Moreover, for each $U$ and for each $b$ in $U$, we have the composition 
\[ \xymatrix{ b \times F \ar[r]^-{\varphi_U \vert_b} & p^{-1}(b) \ar[r]^-{i_b} & p^{-1}(U),} \]
where $i_b$ is the inclusion, and we have the induced composition
\[ \xymatrix{ \Omega(b \times F) \ar[r]^-{(\varphi_u \vert_b)^*} & \Omega(p^{-1}(b)) \ar[r]^-{(i_b)^*} & \Omega(p^{-1}(U)).} \]

Let $[w]$ be an element in $\Omega^r(\mathcal{F})$ such that $d_{\mathcal{F}} [w] =0$. To prove the theorem, we need to prove that there exists an element $[\tau]$ in $\Omega^{r-1}(\mathcal{F})$ such that $d_{\mathcal{F}} [\tau] = [w]$; however, this claim is equivalent to the following statements:
\begin{itemize}
\item[(\textit{i})] $d \tau - w$ belongs to $\Omega^r(M)$.
\item[(\textit{ii})] $(d \tau- w)_m(u_1, \dots, u_r) =0$ for each $b \in B$, for all  $m \in p^{-1}(b)$ and for all $u_1, \dots, u_r \in T_m p^{-1}(b)$. 
\item[(\textit{iii})] $i_b^*(  (d \tau -w)\vert_{p^{-1}(U)}  )  = 0$  for each $b \in U$ and for each $U$ in the good cover.
\item[(\textit{iv})] For each $b \in U$ and for each $U$ in the good cover, 
\begin{equation}\label{beta2}
i_b^*(d \tau \vert_{p^{-1}(U)}) = i_b^*(w \vert_{p^{-1}(U)}).
\end{equation}
\end{itemize}
%$\iff$ $(\varphi_U \vert_b)^* i_b^*(d \tau \vert_{p^{-1}(U)}) = (\varphi_U \vert_b)^*i_b^*(w \vert_{p^{-1}(U)})$  
%for each $b \in U$ and for each $U$ in the good cover. 
%$\iff$ there exists $\tau \in \Omega^{r-1}(M)$ such that 
%$$d(\varphi_U \vert_b)^* i_b^*( \tau \vert_{p^{-1}(U)})  = (\varphi_U \vert_b)^*i_b^*(w \vert_{p^{-1}(U)})$$  
%for each $b \in U$ and for each $U$ in the good cover. 
%\medskip
We will prove (\textit{iv}). By hypothesis we have that $d_{\mathcal{F}}[w] = 0$, then $i_b^*(w \vert_{p^{-1}(U)})$ is closed for each $b \in U$ and for each $U$ in the good cover. On the other hand, also by hypothesis all closed forms in $\Omega(p^{-1}(b))$ are exact for each $b \in B$. This happens, however, if and only if all closed forms in $\Omega(b \times F)$ are exact for each  $b \in U$ and all $U$ in the good cover; this is a consequence of the isomorphism $(\varphi\vert_U)^* : \Omega(p^{-1}(b)) \to \Omega(b \times F)$, commutes with $d$.

Therefore, we get that $(\varphi_U \vert_b)^*i_b^*(w \vert_{p^{-1}(U)}) \in \Omega(b \times F)$ is closed for each $b \in U$ and all $U$ in the good cover. Thus, $(\varphi_U \vert_b)^*i_b^*(w \vert_{p^{-1}(U)}) \in \Omega(b \times F)$ is exact for each $b \in U$ and all $U$ in the good cover.

Since for each fixed $U$ in the good cover the family
\[ \{ (\varphi_U \vert_b)^*i_b^*(w \vert_{p^{-1}(U)})  \}_{b \in U} \]
is a smooth family of exact $(r-1)$-forms on $F$ depending on the parameter $b$, we have by Theorem \ref{AAA} that there exists a smooth family 
\[ \{ \tau_b \}_{b\in U}\]
of $(r-1)$-forms on $F$ depending on the parameter $b$ such that 
\[ d\tau_b = (\varphi_U \vert_b)^*i_b^*(w \vert_{p^{-1}(U)})\]
for each $b \in U$.

On the other hand, we have the following commutative diagram:
\[ \xymatrix{ 
B \times F																		&				\\
U \times F\ar[u] \ar[r]^-{\varphi_U} 											&	p^{-1}(U)	\\
b \times F \ar[u]_-{\widetilde{i}_b} \ar[r]^-{\varphi_U \vert_{b \times F}}	&	p^{-1}(b) \ar[u]^-{i_b},
} \]
where $\widetilde{i}_b$ is the inclusion. Such a diagram induces the following commutative diagram:
\begin{gather}
\begin{aligned}
\xymatrix{
\Omega(B \times F) \ar[d]							&																	\\
\Omega(U \times F) \ar[d]_-{\widetilde{i}_b^*}	&	\Omega(p^{-1}(U)) \ar[l]_-{\varphi_U^*} \ar[d]^-{i_b^*}		\\
\Omega(b \times F)									&	\Omega(p^{-1}(b)) \ar[l]^-{(\varphi_U \vert_{b \times F})^*}.
}
\end{aligned}
\label{alfa}
\end{gather}

Let $\{ \rho_\alpha : M\rightarrow \mathbb{R}   \}_{\alpha \in J}$ be a partition of unity subordinated to the good cover $\{  U_{\alpha}  \}_{\alpha \in J}$ of $B$. We define $\tau \in \Omega^{r-1}(B \times F)$ as follows: for each $(b,f) \in B \times F$ and each $u_1, \dots, u_{r-1} \in T_{(b,f)}(B \times F)$ we set 
\[ \tau_{(b,f)}(u_1, \dots, u_{r-1}) = \sum_{\alpha \in J} \rho_\alpha(b) (\tau_b)_{(b,f)}( d\pi_1 u_1, \dots, d\pi_1 u_{r-1}). \]

Now we will prove first that $\widetilde{i_b}^*(\tau \vert_{U \times F}) \in \Omega^{r-1}(b \times F)$ is such that
\[ d \widetilde{i_b}^*(\tau \vert_{U \times F}) = (\varphi_U \vert_{b \times F})^* i_b^*(w \vert_{p^{-1}(U)}) \]
for each $U$ in the good cover and for each $b \in U$.

For each $U \in \{ U_\alpha \}_{\alpha \in J}$, $b\in U$ and $u_1, \dots, u_{r-1} \in T_{(b,f)}(b \times F)$, we get
\begin{align*}
(d \widetilde{i_b}^*(\tau \vert_{U \times F}))_{(b,f)}(u_1, \dots, u_{r-1}) & = d \tau_{\widetilde{i}_b (b,f)}(d \widetilde{i}_b (u_1), \dots, d \widetilde{i}_b (u_{r-1}) ) \\
& = d \sum_{\alpha \in J} \rho_\alpha(b)(\tau_b)_{(b,f)} (d \pi_1 d \widetilde{i}_b (u_1), \dots, d\pi_1 d \widetilde{i}_b (u_{r-1}) ) \\
& = \sum_{\alpha \in J} \rho_\alpha(b) d (\tau_b)_{(b,f)}(u_1, \dots, u_{r-1} ) \\
& = \sum_{\alpha \in J} \rho_\alpha(b) ( (\varphi_U \vert_{b \times F})^* i_b^*(w \vert_{p^{-1}(U)})  )_{(b,f)}(u_1, \dots, u_{r-1} ).
\end{align*}

Therefore, for each $(b,f) \in b \times F$ we have that
\begin{align*}
(d \widetilde{i_b}^*(\tau \vert_{U \times F}))_{(b,f)} & = \sum_{\alpha \in J} \rho_\alpha(b) ( (\varphi_U \vert_{b \times F})^* i_b^*(w \vert_{p^{-1}(U)})  )_{(b,f)} \\
& = 1 \cdot ( (\varphi_U \vert_{b \times F})^* i_b^*(w \vert_{p^{-1}(U)})  )_{(b,f)}.
\end{align*}

That means
\begin{equation}\label{alfa1}
d \widetilde{i_b}^*(\tau \vert_{U \times F}) = (\varphi_U \vert_{b \times F})^* i_b^*(w \vert_{p^{-1}(U)}) 
\end{equation} 
for each $U$ and for each $b \in U$ as was established.

Now since diagram (\ref{alfa}) commutes, we have 
\begin{equation}\label{alfa2}
\widetilde{i}_b^* =(\varphi_U \vert_{b \times F})^* i_b^*(\varphi_U^{-1})^*,
\end{equation}
and by substituting equation (\ref{alfa2}) in equation (\ref{alfa1}) we get
\begin{align*}
(\varphi_U \vert_{b \times F})^* i_b^*(w \vert_{p^{-1}(U)}) & = d  (\varphi_U \vert_{b \times F})^* i_b^*(\varphi_U^{-1})^*  (\tau \vert_{U \times F}) \\
& = (\varphi_U \vert_{b \times F})^* d i_b^*(\varphi_U^{-1})^*  (\tau \vert_{U \times F}).
\end{align*}

Moreover, since $(\varphi_U \vert_{b \times F})^*$ is an isomorphism we get 
\[ i_b^*(w \vert_{p^{-1}(U)}) = d i_b^*(\varphi_U^{-1})^*  (\tau \vert_{U \times F}) \]
for each $U$ and each $b \in U$.

Now we define the element $\tau_U \in \Omega(p^{-1}(U))$ by
\[ \tau_U =  (\varphi_U^{-1})^*  (\tau \vert_{U \times F}),\]
so we get
\begin{equation}\label{beta}
i_b^*(w \vert_{p^{-1}(U)}) = d i_b^*(\tau_U) 
\end{equation}
for each $U$ and each $b \in U$.

On the other hand, for each $U$, $V$ in $\{ U_\alpha \}$,
\begin{align*}
\tau_U \vert_{ p^{-1}(U) \cap p^{-1}(V) } - \tau_V \vert_{ p^{-1}(U) \cap p^{-1}(V) } & = ( \ (\varphi_U^{-1})^*  (\tau \vert_{U \times F}) \ ) \vert_{ p^{-1}(U) \cap p^{-1}(V)   } - ( \ (\varphi_V^{-1})^*  (\tau \vert_{V \times F}) \ ) \vert_{ p^{-1}(U) \cap p^{-1}(V)   } \\
& = (\varphi_{U \cap V}^{-1})^*  (\tau \vert_{U\cap V \times F}) - (\varphi_{U \cap V}^{-1})^*  (\tau \vert_{U\cap V \times F}) \\
& = 0.
\end{align*}

Therefore, the collection $\{ \tau_U \}$ with $U$ in $\{ U_\alpha \}_{\alpha \in J}$ defines a global form $\widetilde{\tau} \in \Omega^{r-1}(M)$ such that $\widetilde{\tau}\vert_{p^{-1}(U)} = \tau_U$. Thus, by equation (\ref{beta}), we have that
\[ i_b^*(w \vert_{p^{-1}(U)}) = d i_b^*(\tau \vert_{p^{-1}(U)}) = i_b^*(d\tau \vert_{p^{-1}(U)}) \]
for each $U$ in the good cover and for each $b \in U$. Such element satisfies equation (\ref{beta2}), and we conclude with the proof of the theorem.
\end{proof}

\section*{Acknowledgements}
The author Rafael Ramos acknowledges Dr. Yuri Vorobiev for the useful discussions and advice. The research was partially supported by a CONACyT research grant.

\bibliographystyle{abbrv}

\end{document}